\documentclass[a4paper]{amsart}

\usepackage{graphicx}
\graphicspath{{./figure/}}

\usepackage{color,here,url,hyperref,bm}

\makeatletter
\@namedef{subjclassname@2020}{\textup{2020} Mathematics Subject Classification}
\makeatother

\theoremstyle{plain}
 \newtheorem{thm}{Theorem}[section]
 \newtheorem{lem}[thm]{Lemma}
 \newtheorem{prop}[thm]{Proposition}
 \newtheorem{cor}[thm]{Corollary}

\theoremstyle{definition}
 
\theoremstyle{remark}
 \newtheorem{rem}[thm]{Remark}

\newcommand{\fig}[3][width=12cm]{
\begin{figure}[htbp]
 \centering 
 \includegraphics[#1,clip]{#2} 
 \caption{#3} 
\label{fig:#2}
\end{figure}}

\newcommand{\Z}{\mathbb{Z}}

\begin{document}

\title[Degeneration of hyperbolic cone structures]{Degeneration of 3-dimensional hyperbolic cone structures with decreasing cone angles}
\author{Ken'ichi YOSHIDA}
\address{Center for Soft Matter Physics, Ochanomizu University, 2-1-1 Ohtsuka, Bunkyo-ku, Tokyo 112-8610, Japan}
\email{yoshida.kenichi@ocha.ac.jp}
\date{}

\begin{abstract}
For 3-dimensional hyperbolic cone structures with cone angles $\theta$, 
local rigidity is known for $0 \leq \theta \leq 2\pi$, 
but global rigidity is known only for $0 \leq \theta \leq \pi$. 
The proof of the global rigidity by Kojima is based on the fact 
that hyperbolic cone structures with cone angles at most $\pi$ 
do not degenerate in deformations decreasing cone angles to zero. 

In this paper, 
we give an example of a degeneration of hyperbolic cone structures 
with decreasing cone angles less than $2\pi$. 
These cone structures are constructed on a certain alternating link in the thickened torus 
by gluing four copies of a certain polyhedron. 
For this construction, 
we explicitly describe the isometry types on such a hyperbolic polyhedron. 
\end{abstract}

\maketitle

\section{Introduction}
\label{section:intro}

A 3-dimensional hyperbolic cone-manifold is a hyperbolic 3-manifold 
with cone-type singularities. 
In this paper, 
we assume that a cone-manifold has finite volume, 
and cone singularities consist of disjoint closed geodesics.

Let $X$ be a 3-manifold, 
and let $\Sigma$ be a link in $X$. 
Let $\Sigma_{1}, \dots , \Sigma_{n}$ denote the components of $\Sigma$. 
Suppose that $X \setminus \Sigma$ admits an incomplete hyperbolic structure, 
and the completed metric has the form 
\[
dr^{2} + \sinh^{2} r d \theta^{2} + \cosh^{2} r dz^{2}
\]
in cylindrical coordinates around each component $\Sigma_{i}$ for $1 \leq i \leq n$, 
where $r$ is the distance from the singular locus, 
$z$ is the distance along the singular locus, 
and $\theta$ is the angle measured modulo $\theta_{i} > 0$. 
Then the metric on $(X, \Sigma)$ is called a \textit{hyperbolic cone structure}. 
More precisely, an equivalence class of such cone metrics by isometries isotopic to the identity 
is a cone structure. 
The angle $\theta_{i}$ is called the \textit{cone angle} at the cone locus $\Sigma_{i}$. 
If $\theta_{i} = 2\pi$, the cone locus $\Sigma_{i}$ can be regarded as non-singular. 
Furthermore, a cusp of a hyperbolic 3-manifold can be regarded as 
a cone locus with cone angle zero. 
This is justified by the fact that 
hyperbolic cone structures 
converge to a cusped hyperbolic structure 
in the pointed Gromov--Hausdorff topology 
if the cone angles converge to zero.

From now on, fix a pair $(X, \Sigma)$. 
Let $\mathcal{C}$ denote the set of cone structures on $(X, \Sigma)$ 
such that the cone angles are at most $2\pi$. 
Suppose that there is a finite volume cone structure $g \in \mathcal{C}$. 
Then any $g \in \mathcal{C}$ has finite volume. 
The set $\mathcal{C}$ admits the pointed Gromov--Hausdorff topology, 
which is induced by the geometric convergence of metric spaces. 
The continuous map $\Theta \colon \mathcal{C} \to [0,2\pi]^{n}$ 
is defined by $\Theta (g) = (\theta_{1}, \dots , \theta_{n})$, 
where $\theta_{i}$ is the cone angle at $\Sigma_{i}$ in the cone-manifold $(X, \Sigma; g)$. 
Let $g_{0}$ be an element in $\mathcal{C}$ such that 
$\Theta (g_{0}) = (0, \dots , 0)$. 
The cusped hyperbolic structure $g_{0}$ is unique for $(X, \Sigma)$ 
by the Mostow--Prasad rigidity~\cite{mostow1973strong, prasad1973strong}.

Local and global rigidity for hyperbolic cone structures 
are known as follows. 

\begin{thm}[The local rigidity by Hodgson and Kerckhoff~\cite{hodgson1998rigidity}]
\label{thm:local}
The space $\mathcal{C}$ is Hausdorff, 
and the map $\Theta \colon \mathcal{C} \to [0,2\pi]^{n}$ is a local homeomorphism. 
In other words, the space $\mathcal{C}$ is locally parametrized by the cone angles. 
\end{thm}

The local rigidity does not hold in general if cone angles exceed $2\pi$. 
Izmestiev~\cite{izmestiev2011examples} 
constructed infinitesimally flexible hyperbolic cone-manifolds 
with cone angles more than $2\pi$ 
by gluing polyhedra.

\begin{thm}[The global rigidity by Kojima~\cite{kojima1998deformations}]
\label{thm:global}
Let 
\[
\mathcal{C}_{[0,\pi]} = \{ g \in \mathcal{C} \mid \Theta (g) \in [0,\pi]^{n} \}. 
\] 
Then the map $\Theta |_{\mathcal{C}_{[0,\pi]}} \colon \mathcal{C}_{[0,\pi]} \to [0,\pi]^{n}$ 
is injective. 
In other words, the cone structure is determined by the cone angles 
if the cone angles do not exceed $\pi$. 
\end{thm}

The global rigidity is not known 
when cone angles exceed $\pi$. 
Theorem~\ref{thm:global} follows from 
the Mostow--Prasad rigidity for $g_{0}$ 
and Theorems~\ref{thm:local} and \ref{thm:decrease}.

\begin{thm}[Kojima~\cite{kojima1998deformations}]
\label{thm:decrease}
Let $g \in \mathcal{C}$. 
Suppose that $\Theta (g) = (\theta_{1}, \dots , \theta_{n}) \in [0,\pi]^{n}$. 
Then there is $A \subset \mathcal{C}$ such that 
$g \in A$ and 
$\Theta |_{A} \colon A \to [0,\theta_{1}] \times \dots \times [0,\theta_{n}]$ is a homeomorphism. 
In other words, 
we can obtain a continuous family of cone structures from $g$ to $g_{0}$ 
by arbitrarily decreasing cone angles. 
\end{thm}

Similar results are known for 3-dimensional hyperbolic cone-manifolds with vertices. 
The local rigidity for cone angles less than $2\pi$ was proved 
by Mazzeo and Montcouquiol~\cite{mazzeo2011infinitesimal}, 
and independently Weiss~\cite{weiss2013deformation}. 
The global rigidity for cone angles at most $\pi$ was proved 
by Weiss~\cite{weiss2007global}.

Theorem~\ref{thm:main} is the main result of this paper. 
It implies that 
Theorem~\ref{thm:decrease} cannot be generalized 
for cone angles less than $2\pi$. 
A \textit{continuous degenerating family} of cone structures on $(X, \Sigma)$ 
is a continuous map $\gamma \colon [0,1) \to \mathcal{C}$ such that 
$\lim_{x \to 1} \Theta (\gamma (x)) \in [0,2\pi]^{n}$ 
but $\gamma (x)$ does not converge in $\mathcal{C}$ as $x \to 1$.

\begin{thm}
\label{thm:main}
Let $L = L_{1} \sqcup \dots \sqcup L_{4} \subset T^{2} \times I$ 
be a link in the thickened torus as indicated in Figure~\ref{fig:dcda-weave.pdf}, 
where $I$ is an open interval. 
Then there is a continuous degenerating family of cone structures on $(T^{2} \times I, L)$ 
with decreasing cone angles. 
In this degeneration, 
two of the cone loci $L$ intersect transversally. 
Two simultaneous intersections may occur. 
\end{thm}

\fig[width=6cm]{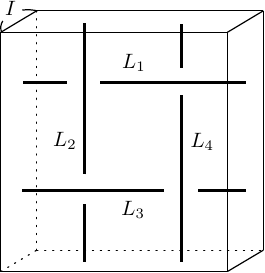}{A link $L$ in $T^{2} \times I$. (A fundamental domain of $T^{2} \times I$ is drawn.)}

Links in $T^{2} \times I$ have been studied in several situations. 
One of them concerns 
the hyperbolic structures on the complements. 
By projecting a link in $T^{2} \times I$ to $T^{2}$, 
we obtain a diagram of the link in $T^{2}$. 
The diagram gives the notion of alternating links in $T^{2} \times I$. 
In \cite{adams2020generalized, champanerkar2019geometry}, 
the hyperbolic structure on the complement of an alternating link in $T^{2} \times I$ 
is constructed by gluing ideal bipyramids given from the diagram. 
The above $L$ is one of the simplest examples of alternating links, 
and it was described in detail 
by Champanerkar, Kofman, and Purcell~\cite{champanerkar2016geometrically}.

We will construct cone structures on $(T^{2} \times I, L)$ 
by polyhedral decomposition, 
which is a generalization of the above construction for cusped hyperbolic structures. 
If a hyperbolic cone structure on $(T^{2} \times I, L)$ has sufficient symmetry, 
the hyperbolic cone-manifold can be decomposed into 
four copies of a certain polyhedron, 
called a tetragonal trapezohedron (a.k.a. an antibipyramid). 
A tetragonal trapezohedron is the dual of a square antiprism. 
Thus we will be reduced to considering 
isometric types of a hyperbolic tetragonal trapezohedron.

We need to know whether a tetragonal trapezohedron with the assigned dihedral angles 
can be realized in the hyperbolic space. 
The most powerful result for our problem is Andreev's theorem~\cite{andreev1970convex}, 
which gives 
the condition by linear inequalities of dihedral angles 
for a finite volume hyperbolic polyhedron with non-obtuse dihedral angles 
(see \cite{roeder2007andreev} for details of compact cases). 
For obtuse dihedral angles, however, 
Diaz~\cite{diaz1997non} gave an example that 
no linear inequalities of dihedral angles hold. 
We will obtain another such example. 
Since there are no general tools for our problem, 
we need to explicitly describe isometric types of a tetragonal trapezohedron.

\section*{Acknowledgments} 
This work is supported by 
JSPS KAKENHI Grant Numbers 15H05739 and 19K14530, 
and JST CREST Grant Number JPMJCR17J4.

\section{An alternating link in the thickened torus}
\label{section:link}

We consider a link  $L = L_{1} \sqcup \dots \sqcup L_{4} \subset T^{2} \times I$ 
as indicated in Figure~\ref{fig:dcda-weave.pdf}. 
Let $\mathcal{C}$ denote the space of cone structures on $(T^{2} \times I, L)$ 
as in Section~\ref{section:intro}, 
where the components of $T^{2} \times \partial I$ keep to be two cusps. 
Note that any of the cone angles cannot be equal to $2\pi$. 
For instance, if the cone angle at $L_{1}$ is equal to $2\pi$, 
the cone loci $L_{2}$ and $L_{4}$ are parallel to a cusp, 
which is impossible. 
The map $\Theta \colon \mathcal{C} \to [0,2\pi)^{4}$ 
assigns the cone angles at $L_{i}$. 
If each $L_{i}$ is a cusp, 
we obtain $g_{0} \in \mathcal{C}$, 
which is the unique element of $\mathcal{C}$ 
satisfying that $\Theta (g_{0}) = (0, \dots ,0)$. 
Let $\mathcal{C}_{0} \subset \mathcal{C}$ denote the component containing $g_{0}$.

We consider symmetry of $(T^{2} \times I, L)$. 
There is an automorphism $\gamma_{1}$ of $(T^{2} \times I, L)$ 
that fixes each point of $L_{1}$ and $L_{3}$, 
and fixes each of $L_{2}$ and $L_{4}$ as sets, 
but reverses the orientations of $L_{2}$ and $L_{4}$. 
The fixed point set of $\gamma_{1}$ is two open annuli containing $L_{1}$ and $L_{3}$, 
whose ends are contained in neighborhood of $T^{2} \times \partial I$. 
By changing the roles of $(L_{1}, L_{3})$ and $(L_{2}, L_{4})$, 
we obtain an automorphism $\gamma_{2}$ of $(T^{2} \times I, L)$. 
The automorphisms $\gamma_{1}$ and $\gamma_{2}$ are determined up to isotopy. 
Let $\Gamma$ denote the group generated by $\gamma_{1}$ and $\gamma_{2}$. 
We choose $\gamma_{1}$ and $\gamma_{2}$ so that 
the group $\Gamma$ is isomorphic to $\Z / 2\Z \times \Z / 2\Z$.

We call a cone structure $g \in \mathcal{C}$ \textit{symmetric} 
if the $\Gamma$-action on $(T^{2} \times I, L; g)$ is isotopic to an isometric action. 
Let $\mathcal{C}_{\mathrm{sym}}$ denote the set of symmetric cone structures in $\mathcal{C}$.

\begin{prop}
\label{prop:sym}
The component $\mathcal{C}_{0}$ is contained in $\mathcal{C}_{\mathrm{sym}}$. 
\end{prop}
\begin{proof}
The Mostow--Prasad rigidity implies that $g_{0} \in \mathcal{C}_{\mathrm{sym}}$. 
The local rigidity implies that 
the set $\mathcal{C}_{\mathrm{sym}}$ is closed and open subset of $\mathcal{C}$. 
Hence $\mathcal{C}_{\mathrm{sym}}$ is the union of components of $\mathcal{C}$, 
one of which is $\mathcal{C}_{0}$. 
\end{proof}

In fact, the space $\mathcal{C}_{\mathrm{sym}}$ is connected. 
Hence $\mathcal{C}_{0} = \mathcal{C}_{\mathrm{sym}}$. 
This will be shown in Corollary~\ref{cor:connected}.

In this paper, we cannot treat non-symmetric cone structures. 
We do not even know whether there exists a non-symmetric cone structure. 
It is a really surprising example if it exists. 
Nonetheless, if the global rigidity for $\mathcal{C}$ holds, 
then $\mathcal{C}_{\mathrm{sym}} = \mathcal{C}$. 
This will be shown in Corollary~\ref{cor:global}.

For $g \in \mathcal{C}_{\mathrm{sym}}$ and the isometric $\Gamma$-action, 
the quotient space $(T^{2} \times I, L; g) / \Gamma$ is isometric to 
a (tetragonal) trapezohedron in the hyperbolic 3-space 
as indicated in Figure~\ref{fig:dcda-trapezohedron.pdf}. 
The edge $\hat{L}_{i}$ is the image of the cone locus $L_{i}$. 
The two ideal vertices disjoint from $\hat{L}_{i}$ 
correspond to the components of $T^{2} \times \partial I$. 
The faces are totally geodesic. 
If $\Theta (g) = (\theta_{1}, \dots , \theta_{4})$, 
the dihedral angles are $\pi/2$ 
except the four angles $\alpha_{i} = \theta_{i}/2$ at $\hat{L}_{i}$. 
We remark that $\hat{L}_{i}$ degenerates to an ideal vertex if $\alpha_{i} = 0$. 
We use the term ``trapezohedron'' also for such a degenerated polyhedron. 
Thus we decompose a symmetric cone-manifold $(T^{2} \times I, L; g)$ into four trapezohedra. 
The four trapezohedra correspond to the complementary regions 
of the diagram of $L$ in $T^{2}$ in Figure~\ref{fig:dcda-weave.pdf}. 

Conversely, we can obtain a symmetric hyperbolic cone structure in $\mathcal{C}_{\mathrm{sym}}$
by gluing four trapezohedra. 
The way of gluing is as follows. 
Let $T$ be a hyperbolic trapezohedron 
with right dihedral angles except at $\hat{L}_{i}$. 
We color the faces of $T$ black and white 
in a ``checkerboard'' fashion 
so that two faces with a common color are adjacent only along $\hat{L}_{i}$. 
Take four copies $T_{00}, T_{01}, T_{10}, T_{11}$ of $T$. 
For $j=0,1$, glue $T_{j0}$ and $T_{j1}$ along the black faces, 
and glue $T_{0j}$ and $T_{1j}$ along the white faces. 
Here corresponding vertices are matched. 
This construction can be called ``double of double''. 

\fig[width=6cm]{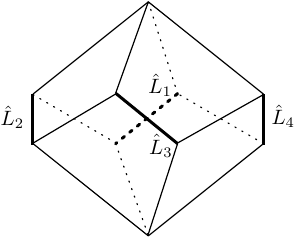}{A tetragonal trapezohedron}

The above argument for $g_{0}$ gives a decomposition of $T^{2} \times I \setminus L$ 
into four regular ideal octahedra. 
This decomposition was given 
in \cite{adams2020generalized, champanerkar2016geometrically}, 
and the ``double of double'' construction was described in detail 
in \cite{kolpakov2013hyperbolic}. 
We remark that $T^{2} \times I \setminus L$ is homeomorphic to 
the complement of the minimally twisted 6-chain link.

We summarize the above argument in the following proposition. 

\begin{prop}
\label{prop:decomp}
There is a natural one-to-one correspondence between $\mathcal{C}_{\mathrm{sym}}$ 
and the set of hyperbolic trapezohedra 
with right dihedral angles except at $\hat{L}_{i}$. 
\end{prop}

\section{Dihedral angles of a tetragonal trapezohedron}
\label{section:angle}

We consider hyperbolic trapezohedra 
(possibly $\hat{L}_{i}$ degenerates to an ideal vertex) 
with right dihedral angles except at $\hat{L}_{i}$. 
Let $\mathcal{A}$ denote the image of 
the map $\frac{1}{2} \Theta \colon \mathcal{C}_{\mathrm{sym}} \to [0,\pi)^{4}$. 
In other words, 
$(\alpha_{1}, \dots , \alpha_{4}) \in \mathcal{A}$ if and only if 
there exists a trapezohedron 
with dihedral angles $\alpha_{i}$ at $\hat{L}_{i}$ and $\pi/2$ at the other edges 
in the hyperbolic space.

\begin{thm}
\label{thm:rigidity}
The isometry class of a hyperbolic trapezohedron is determined by the element of $\mathcal{A}$. 
In other words,  
$\frac{1}{2} \Theta \colon \mathcal{C}_{\mathrm{sym}} \to \mathcal{A}$ is injective. 
\end{thm}

The local rigidity for $\mathcal{C}$ implies that 
$\mathcal{A} \subset [0,\pi)^{4}$ is an open subset. 
Define $\cos \colon [0,\pi)^{4} \to (-1,1]^{4}$ 
by $\cos (\alpha_{1}, \dots , \alpha_{4}) = (\cos \alpha_{1}, \dots , \cos \alpha_{4})$, 
which is a homeomorphism. 
We will often write $c_{i} = \cos \alpha_{i}$. 
Let $\mathcal{A}^{\prime} = \cos (\mathcal{A}) \subset (-1,1]^{4}$. 
We explicitly describe $\mathcal{A}^{\prime}$ instead of $\mathcal{A}$. 
From now on, the indices $i=1, \dots ,4$ are regarded modulo 4.

\begin{thm}
\label{thm:bound}
For $1 \leq i \leq 4$, let a function $\Phi_{i}$ be defined by 
\begin{align*}
\Phi_{i}(c_{1}, \dots , c_{4}) 
& = c_{i}c_{i+1}(c_{i}c_{i+1}+1)c_{i+2}c_{i+3} 
- c_{i}c_{i+1}(c_{i}+c_{i+1})(c_{i+2}+c_{i+3}) \\
& \quad + (c_{i}+c_{i+1})^{2} - c_{i}c_{i+1} -1. 
\end{align*}
Let $\partial \mathcal{A}^{\prime}$ denote 
the frontier of $\mathcal{A}^{\prime}$ in $(-1,1]^{4}$. 
Then 
\[
\mathcal{A}^{\prime} = 
\{ (c_{1}, \dots , c_{4}) \in  (-1,1]^{4} \mid 
\text{for any} \ i, \Phi_{i}(c_{1}, \dots , c_{4}) < 0 \ \text{or} \ c_{i}+c_{i+1} > 0 \}, 
\] 
and we can write 
$\partial \mathcal{A}^{\prime} = \bigcup_{i} \partial_{i} \mathcal{A}^{\prime}$, 
where  
\begin{align*}
\partial_{i} \mathcal{A}^{\prime} = 
\{ (c_{1}, \dots , c_{4}) \in  (-1,1]^{4} 
& \mid \Phi_{i}(c_{1}, \dots , c_{4}) = 0, \quad c_{i}+c_{i+1} \leq 0, \\
& \quad c_{i} \leq c_{i+2}, \quad c_{i+1} \leq c_{i+3} \}. 
\end{align*}
As $(c_{1}, \dots , c_{4}) \in \mathcal{A}^{\prime}$ 
approaches to $\partial_{i} \mathcal{A}^{\prime}$, 
the edge between $\hat{L}_{i}$ and $\hat{L}_{i+1}$ degenerates.
In particular, $\mathcal{A}^{\prime} \neq (-1,1]^{4}$. 
\end{thm}

\begin{rem}
\label{rem:bound}
Clearly $(1,1,1,1) \in \mathcal{A}^{\prime}$. 
Since $\Phi_{i}(1,1,1,1) = 0$, we need the condition $c_{i}+c_{i+1} > 0$ 
in the description of $\mathcal{A}^{\prime}$. 
\end{rem}

We consider a hyperbolic trapezohedron $T$ 
whose dihedral angles are $\alpha_{i}$ at $\hat{L}_{i}$ and $\pi/2$ at the other edges. 
We use the upper half-space model of hyperbolic 3-space. 
Regard $\partial \mathbb{H}^{3} = \mathbb{R}^{2} \cup \{\infty\}$. 
The trapezohedron $T$ has two ideal vertices disjoint from $\hat{L}_{i}$. 
We set them at $\infty$ and $O=(0,0)$. 
We project $T$ to $\mathbb{R}^{2} \subset \partial \mathbb{H}^{3}$ 
as indicated in Figure~\ref{fig:dcda-projection.pdf}. 
The endpoints $\widetilde{P}_{i}$ and $\widetilde{Q}_{i}$ of the edge $\hat{L}_{i}$ 
are projected respectively to $P_{i}$ and $Q_{i}$. 
The images of the faces of $T$ adjacent to $O$ are four quadrilaterals $OQ_{i-1}P_{i}Q_{i}$. 
Their union is a rectangle $P_{1}P_{2}P_{3}P_{4}$. 
If $\alpha_{i} =0$, then $\hat{L}_{i} = P_{i} = Q_{i}$. 

Let $F_{i}$ denote the face of $T$ projected to $OQ_{i-1}P_{i}Q_{i}$. 
We extend $F_{i}$ to a totally geodesic plane $\widetilde{C}_{i}$, 
and let $C_{i}$ denote the boundary of $\widetilde{C}_{i}$. 
By considering the dihedral angles at $F_{i}$, 
we see that the circle $C_{i}$ is orthogonal to $P_{i-1}P_{i}$, 
and the angle between $C_{i}$ and $P_{i}P_{i+1}$ is $\alpha_{i}$ 
as indicated in Figure~\ref{fig:dcda-projection.pdf}.  
Since the dihedral angles at the edges of $T$ around $O$ are $\pi/2$, 
the circles $C_{i}$ and $C_{i+1}$ intersect orthogonally. 
Let $R_{i}$ denote the center of the circle $C_{i}$. 
Then $R_{i}$ is contained in the line $P_{i-1}P_{i}$. 
The segments $OR_{i}$ and $OR_{i+1}$ are orthogonal. 
Let $S_{i} = C_{i} \cap C_{i+1} \setminus O$. 
Then $Q_{i}$ is the intersection of the segments $OS_{i}$ and $P_{i}P_{i+1}$.

\fig[width=12cm]{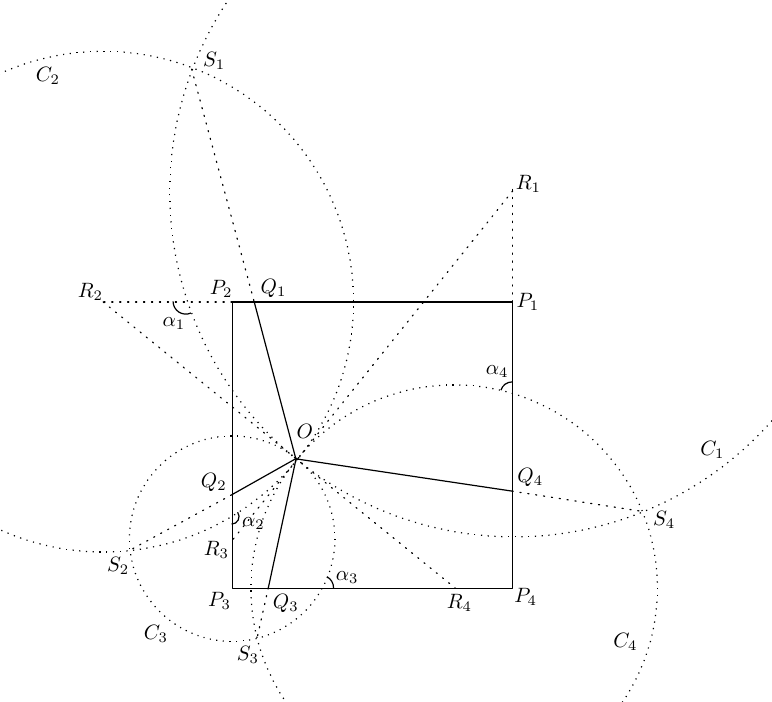}{Projection of a trapezohedron}

Conversely, take points $P_{i}$ and $R_{i}$ in $\mathbb{R}^{2}$ such that
$P_{1}P_{2}P_{3}P_{4}$ is a rectangle containing $O$, and 
$R_{i}$ is contained in the line $P_{i-1}P_{i}$. 
Let $C_{i}$ and $S_{i}$ be as above. 
Then the condition that the projection of a trapezohedron is obtained 
is as follows: 
\begin{itemize}
\item The segments $OS_{i}$ and $P_{i}P_{i+1}$ intersect, and 
\item their intersection $Q_{i}$ is distinct from $P_{i+1}$. 
\end{itemize}
The vertices $\widetilde{P}_{i}$ and $\widetilde{Q}_{i}$ are the intersection of 
the plane $\widetilde{C}_{i}$ and respectively the lines $\infty P_{i}$ and $\infty Q_{i}$. 
If $Q_{i} = P_{i+1}$, 
the edge between $\hat{L}_{i}$ and $\hat{L}_{i+1}$ degenerates, 
which corresponds to intersection of $L_{i}$ and $L_{i+1}$ in $T^{2} \times I$.

We may assume that 
\[
P_{1} = (p_{1},p_{2}), P_{2} = (-p_{3},p_{2}), 
P_{3} = (-p_{3},-p_{4}), P_{4} = (p_{1},-p_{4}), 
\]
where $p_{i} > 0$. 
Let $t$ denote the slope of the line $OR_{1}$. 
Then 
\[
R_{1} = (p_{1},tp_{1}), R_{2} = (-tp_{2},p_{2}), 
R_{3} = (-p_{3},-tp_{3}), R_{4} = (tp_{4},-p_{4}). 
\]
Let $q_{i} = \dfrac{p_{i+1}}{p_{i}}$. 
Since a positive constant multiple on $\mathbb{R}^{2}$ extends 
an isometry of $\mathbb{H}^{3}$, 
the isometry type of $T$ is determined by $q_{i}$ and $t$.

\begin{lem}
\label{lem:angle}
\[
\cos \alpha_{i} = \frac{q_{i}-t}{\sqrt{1+t^{2}}}. 
\]
\end{lem}
\begin{proof}
The radius of the circle $C_{i}$ is equal to $p_{i}\sqrt{1+t^{2}}$. 
The signed length $P_{i}R_{i}$ is equal to $p_{i+1}-tp_{i}$ 
(it is positive if $R_{i}$ is contained inside of the segment $P_{i-1}P_{i}$). 
Then $\cos \alpha_{i}$ is given by their ratio. 
\end{proof}

\begin{lem}
\label{lem:slope}
The condition that the segments $OS_{i}$ and $P_{i}P_{i+1}$ intersect 
and \\ $Q_{i} \neq P_{i+1}$ 
is equivalent to the following inequalities: 
\[
t \geq \frac{1}{2}(q_{i} - q_{i}^{-1}), \quad 
(1-q_{i}q_{i+1}) t < q_{i} + q_{i+1}. 
\]
\end{lem}
\begin{proof}
We prove it for $i=1$. 
By calculating the coordinates of $S_{1}$ from the ones of $O,R_{1}$, and $R_{2}$, 
we have 
\[
S_{1} = \left( \frac{2p_{1}p_{2}(p_{2}-tp_{1})}{p_{1}^{2}+p_{2}^{2}}, 
\frac{2p_{1}p_{2}(tp_{2}+p_{1})}{p_{1}^{2}+p_{2}^{2}} \right). 
\]
The slope of $OS_{1}$ is equal to 
$\dfrac{tp_{2}+p_{1}}{p_{2}-tp_{1}}$. 
Since $Q_{i} = OS_{1} \cap P_{1}P_{2}$, 
we have 
\[
Q_{1} = \left( \dfrac{p_{2}-tp_{1}}{tp_{2}+p_{1}}p_{2}, p_{2} \right).
\] 
Therefore the condition holds if and only if 
\[
\frac{2p_{1}p_{2}(tp_{2}+p_{1})}{p_{1}^{2}+p_{2}^{2}} \geq p_{2}, \quad 
-p_{3} < \frac{p_{2}-tp_{1}}{tp_{2}+p_{1}}p_{2} \leq p_{1}. 
\]
The first inequality is equivalent to 
\[
t \geq \frac{1}{2}(q_{1} - q_{1}^{-1}). 
\]
The first inequality implies that 
$tp_{2}+p_{1} > 0$. 
Under this condition, 
the right of second inequality is also equivalent to 
\[
t \geq \frac{1}{2}(q_{1} - q_{1}^{-1}), 
\]
and the left of second inequality is equivalent to 
\[
(1-q_{1}q_{2}) t < q_{1} + q_{2}. 
\]
\end{proof}

\begin{rem}
\label{rem:slope}
Suppose that the above condition holds for all $i$. 
Since $\prod_{i=1}^{4} q_{i} = 1$, 
there is $i$ such that $q_{i} \geq 1$. 
Hence $t \geq 0$. 
Thus the second inequality is vacuous if $p_{i} \leq p_{i+2}$. 
\end{rem}

\begin{lem}
\label{lem:parameter}
Let 
\[
\mathcal{B} = 
\{(q_{1}, \dots , q_{4},t) \in \mathbb{R}_{>0}^{4} \times \mathbb{R}_{\geq 0} 
\mid \prod_{i=1}^{4} q_{i} = 1, \
t \geq \frac{1}{2}(q_{i} - q_{i}^{-1})\}.  
\]
Define $f \colon \mathcal{B} \to \mathbb{R}^{4}$ by 
\[
f(q_{1}, \dots , q_{4},t) = 
\left( \frac{q_{1}-t}{\sqrt{1+t^{2}}}, \dots , \frac{q_{4}-t}{\sqrt{1+t^{2}}} \right). 
\]
Then $f$ is injective, and the image of $f$ is $(-1,1]^{4}$. 
\end{lem}
\begin{proof}
For fixed $q_{i}$, the function $f_{i}(t) = \dfrac{q_{i}-t}{\sqrt{1+t^{2}}}$ 
is monotonically decreasing. 
Since $f_{i}\left(\frac{1}{2}(q_{i} - q_{i}^{-1})\right) = 1$ 
and $\lim_{t \to \infty} f_{i}(t) = -1$, 
the image of $f$ is contained in $(-1,1]^{4}$. 
Take any $(c_{1}, \dots , c_{4}) \in (-1,1]^{4}$. 
It is sufficient to show that  
there is a unique element $(q_{1}, \dots , q_{4},t) \in \mathcal{B}$ 
such that $f(q_{1}, \dots , q_{4},t) = (c_{1}, \dots , c_{4})$.

Let $g_{i}(t) = t + c_{i}\sqrt{1+t^{2}}$. 
Since $-1 < c_{i} \leq 1$, the function $g_{i}(t)$ is monotonically increasing. 
Note that $\lim_{t \to \infty} g_{i}(t) = \infty$. 
Let $g(t) = \prod_{i=1}^{4}g_{i}(t)$. 
If some $c_{i}$ is negative, 
then we take $t^{\prime} >0$ to be the maximum of $t$ 
satisfying $g_{i}(t) = 0$ for some $i$. 
Then 
$g(t^{\prime}) = 0$ and $g(t)$ is monotonically increasing for $t \geq t^{\prime}$. 
If no $c_{i}$ is negative, 
then $g(0) = \prod_{i=1}^{4}c_{i} \leq 1$ 
and $g(t)$ is monotonically increasing for $t \geq 0$. 
In both cases, 
there is a unique $t_{0} \geq 0$ 
such that $g(t_{0}) = 1$ and $g_{i}(t_{0}) > 0$. 
By setting $q_{i} = g_{i}(t_{0})$, 
we have $\prod_{i=1}^{4} q_{i} = 1$ 
and $f(q_{1}, \dots , q_{4}, t_{0}) = (c_{1}, \dots , c_{4})$. 
Since $f_{i}(t_{0}) = c_{i} \leq 1 = f_{i}\left(\frac{1}{2}(q_{i} - q_{i}^{-1})\right)$ 
and $f_{i}$ is monotonically decreasing, 
we have $t_{0} \geq \frac{1}{2}(q_{i} - q_{i}^{-1})$. 
Thus we have a unique solution. 
\end{proof}

Let $\mathcal{B}_{0} = \{(q_{1}, \dots , q_{4},t) \in \mathcal{B} 
\mid (1-q_{i}q_{i+1}) t < q_{i} + q_{i+1} \}$. 
Lemma~\ref{lem:slope} implies that 
an element of $\mathcal{B}_{0}$ corresponds to a hyperbolic trapezohedron 
with right dihedral angles except at $\hat{L}_{i}$. 
Therefore $\mathcal{A}^{\prime} = f (\mathcal{B}_{0})$ 
by Lemma~\ref{lem:angle}. 
Now $f \colon \mathcal{B} \to (-1,1]^{4}$ is a homeomorphism. 
Since $\mathcal{B}_{0} \neq \mathcal{B}$, 
we have $\mathcal{A}^{\prime} \neq (-1,1]^{4}$.

\begin{proof}[Proof of Theorem~\ref{thm:rigidity}]
The map $f \colon \mathcal{B}_{0} \to \mathcal{A}^{\prime}$ 
is injective by Lemma~\ref{lem:parameter}. 
Hence the isometry class of a trapezohedron $T$ is determined by the dihedral angles. 
\end{proof}

\begin{proof}[Proof of Theorem~\ref{thm:bound}]
Let $\partial \mathcal{B}_{0}$ be the frontier of $\mathcal{B}_{0}$ in $\mathcal{B}$. 
Then $f (\partial \mathcal{B}_{0}) = \partial \mathcal{A}^{\prime}$, 
and $f (1, \dots , 1, 0) = (1, \dots , 1)$. 
Let us describe $\partial \mathcal{A}^{\prime}$. 
Define 
\[
\partial_{i} \mathcal{A}^{\prime} = 
\{(c_{1}, \dots , c_{4}) \in \partial \mathcal{A}^{\prime} \mid 
c_{i} \leq c_{i+2}, \ c_{i+1} \leq c_{i+3} \}. 
\] 
Then $\partial \mathcal{A}^{\prime} = \bigcup \partial_{i} \mathcal{A}^{\prime}$.

We consider $\partial_{1} \mathcal{A}^{\prime}$. 
Recall that $c_{i} = \cos \alpha_{i} = \dfrac{q_{i}-t}{\sqrt{1+t^{2}}}$. 
Since $c_{1} \leq c_{3}$ and $c_{2} \leq c_{4}$,  
we have $q_{1} \leq q_{3}$ and $q_{2} \leq q_{4}$. 
Since $\prod_{i=1}^{4} q_{i} = 1$, 
we have $q_{1}q_{2} \leq 1 \leq q_{3}q_{4}$. 
If $q_{4}q_{1} < 1$, 
then $\dfrac{q_{1}+q_{2}}{1-q_{1}q_{2}} \leq \dfrac{q_{4}+q_{1}}{1-q_{4}q_{1}}$. 
If $q_{2}q_{3} < 1$, 
then $\dfrac{q_{1}+q_{2}}{1-q_{1}q_{2}} \leq \dfrac{q_{2}+q_{3}}{1-q_{2}q_{3}}$. 
Hence it is sufficient to consider only the condition that 
$t < \dfrac{q_{1}+q_{2}}{1-q_{1}q_{2}}$. 
Substitute 
$q_{i} = t + c_{i} \sqrt{1+t^{2}}$
in $(1-q_{1}q_{2})t < q_{1} + q_{2}$. 
Then it is equivalent to 
\[
(c_{1}c_{2}+1)t > -(c_{1}+c_{2})\sqrt{1+t^{2}}. 
\]
If $c_{1}+c_{2} > 0$, the inequality holds trivially. 
Suppose that $c_{1}+c_{2} \leq 0$. 
Note that $c_{1}, c_{2} \neq 1$. 
Then (by taking squares of both sides) it is equivalent to 
\[
t > \frac{-(c_{1}+c_{2})}{\sqrt{(1-c_{1}^{2})(1-c_{2}^{2})}}, 
\]
which is also equivalent to 
\[
\sqrt{1+t^{2}} > \frac{c_{1}c_{2}+1}{\sqrt{(1-c_{1}^{2})(1-c_{2}^{2})}}. 
\]
Then they are also equivalent to each of 
\begin{align*}
q_{1}q_{2} & = t \left( (c_{1}c_{2}+1)t + (c_{1}+c_{2})\sqrt{1+t^{2}} \right) 
+ c_{1}c_{2} 
> c_{1}c_{2}, \\
q_{3}q_{4} & = t^{2} + (c_{3}+c_{4})t\sqrt{1+t^{2}} + c_{3}c_{4}(1+t^{2}) \\
& > \frac{(c_{1}+c_{2})^{2} - (c_{1}+c_{2})(c_{1}c_{2}+1)(c_{3}+c_{4}) 
+ (c_{1}c_{2}+1)^{2}c_{3}c_{4}}{(1-c_{1}^{2})(1-c_{2}^{2})}. 
\end{align*}
Since $\prod_{i=1}^{4} q_{i} = 1$, 
we have 
\[
c_{1}c_{2}\left((c_{1}+c_{2})^{2} - (c_{1}+c_{2})(c_{1}c_{2}+1)(c_{3}+c_{4}) 
+ (c_{1}c_{2}+1)^{2}c_{3}c_{4}\right) 
< (1-c_{1}^{2})(1-c_{2}^{2}). 
\]
Since $c_{1}c_{2}(c_{1}+c_{2})^{2} - (1-c_{1}^{2})(1-c_{2}^{2}) = 
(c_{1}c_{2}+1)\left((c_{1}+c_{2})^{2} - c_{1}c_{2} - 1 \right)$ and $c_{1}c_{2}+1 > 0$, 
we have 
\begin{align*}
\Phi_{1}(c_{1}, \dots , c_{4}) 
& = c_{1}c_{2}(c_{1}c_{2}+1)c_{3}c_{4} 
- c_{1}c_{2}(c_{1}+c_{2})(c_{3}+c_{4}) \\
& \quad + (c_{1}+c_{2})^{2} - c_{1}c_{2} -1 \\
& < 0. 
\end{align*}

After all, 
$t < \dfrac{q_{1}+q_{2}}{1-q_{1}q_{2}}$ is equivalent to 
$c_{1}+c_{2} > 0$ or $\Phi_{1}(c_{1}, \dots , c_{4}) < 0$.  
Under the condition that $c_{1} \leq c_{3}$ and $c_{2} \leq c_{4}$, 
the frontier of $\mathcal{A}^{\prime}$ is given by 
$\Phi_{1} = 0$ and $c_{1}+c_{2} \leq 0$. 
In this case $Q_{i} = P_{i+1}$, 
which means that the edge between $\hat{L}_{i}$ and $\hat{L}_{i+1}$ degenerates. 
Therefore $\mathcal{A}^{\prime}$ and $\partial_{i} \mathcal{A}^{\prime}$ 
are described as in the assertion. 
\end{proof}

Let us see the shape of $\mathcal{A}$ 
more explicitly. 

\begin{cor}
\label{cor:abab}
For any $\alpha, \beta \in [0,\pi)$, 
it holds that $(\alpha, \beta, \alpha, \beta) \in \mathcal{A}$. 
Consequently, 
$\partial_{i} \mathcal{A}^{\prime} \cap \partial_{i+2} \mathcal{A}^{\prime} = \emptyset$. 
\end{cor}
\begin{proof}
If $(c_{1}, \dots , c_{4}) \in 
\partial_{i} \mathcal{A}^{\prime} \cap \partial_{i+2} \mathcal{A}^{\prime}$, 
then $c_{1} = c_{3}$ and $c_{2} = c_{4}$. 
If $-1 < c_{1}, c_{2} < 1$, 
then $\Phi_{1}(c_{1},c_{2},c_{1},c_{2}) = -(1-c_{1}c_{2})(1-c_{1}^{2})(1-c_{2}^{2}) < 0$. 
By the same argument for all $\Phi_{i}$'s, 
we have $(c_{1},c_{2},c_{1},c_{2}) \in \mathcal{A}^{\prime}$ 
for any $c_{1}, c_{2} \in (-1,1]$. 
\end{proof}

\begin{cor}
\label{cor:double}
It holds that 
$\partial_{i} \mathcal{A}^{\prime} \cap \partial_{i+1} \mathcal{A}^{\prime} \neq \emptyset$. 
In other words, there is a degeneration in $\mathcal{C}$ 
in which two intersections of cone loci occur. 
\end{cor}
\begin{proof}
Since $\Phi_{1}(c,c,c,1) = \Phi_{2}(c,c,c,1) = (c-1)^{2}(c+1)(c^{2}-c-1)$, 
we have 
\[
\left( \dfrac{1-\sqrt{5}}{2}, \dfrac{1-\sqrt{5}}{2}, \dfrac{1-\sqrt{5}}{2}, 1 \right) 
\in \partial_{1} \mathcal{A}^{\prime} \cap \partial_{2} \mathcal{A}^{\prime}. 
\]
\end{proof}

\begin{thm}
\label{thm:cube}
It holds that $[0,\arccos(1-\sqrt{2}))^{4} \subset \mathcal{A}$. 
Furthermore, \\ $(\arccos(1-\sqrt{2}),\arccos(1-\sqrt{2}),0,0) \notin \mathcal{A}$. 
Hence the value $\arccos(1-\sqrt{2})$ is best possible. 
\end{thm}
\begin{proof}
Since $\Phi_{1}(c,c,1,1) = (c-1)^{2}(c^{2}-2c-1)$, 
we have $(1-\sqrt{2},1-\sqrt{2},1,1) \in \partial_{1} \mathcal{A}^{\prime}$. 
Hence $(\arccos(1-\sqrt{2}),\arccos(1-\sqrt{2}),0,0) \notin \mathcal{A}$.

Without loss of generality, 
it is sufficient to show that 
$\Phi_{1}(c_{1}, \dots , c_{4}) < 0$ or $c_{1}+c_{2} > 0$ 
if $1-\sqrt{2} < c_{1} \leq c_{3} \leq 1$ and $1-\sqrt{2} < c_{2} \leq c_{4} \leq 1$. 
Suppose that $1-\sqrt{2} < c_{1} \leq c_{3} \leq 1, 1-\sqrt{2} < c_{2} \leq c_{4} \leq 1$, 
and $c_{1}+c_{2} \leq 0$. 
Note that $c_{1}, c_{2} \neq 1$. 
If $c_{1} = 0$, then $\Phi_{1}(c_{1}, \dots , c_{4}) = c_{2}^{2} - 1 < 0$. 
The same argument holds if $c_{2} = 0$. 
Hence we may assume that $c_{1}c_{2} \neq 0$.

For fixed $c_{1}$ and $c_{2}$, the equation 
\begin{align*}
\Phi_{1}(c_{1}, \dots , c_{4}) 
& = c_{1}c_{2}(c_{1}c_{2}+1) 
\left( c_{3} - \frac{c_{1}+c_{2}}{c_{1}c_{2}+1} \right) 
\left( c_{4} - \frac{c_{1}+c_{2}}{c_{1}c_{2}+1} \right) \\ 
& \quad - \frac{(1-c_{1}^{2})(1-c_{2}^{2})}{c_{1}c_{2}+1} \\
& = 0
\end{align*}
gives a hyperbola $H$ in the $(c_{3},c_{4})$-plane. 
The asymptotic lines of $H$ are given by 
$c_{3} = \dfrac{c_{1}+c_{2}}{c_{1}c_{2}+1}$ and 
$c_{4} = \dfrac{c_{1}+c_{2}}{c_{1}c_{2}+1}$. 
Since $-\dfrac{(1-c_{1}^{2})(1-c_{2}^{2})}{c_{1}c_{2}+1} < 0$, 
the inequality $\Phi_{1}(c_{1}, \dots , c_{4}) < 0$
gives the complementary region of $H$ 
containing the two asymptotic lines. 

Suppose that $c_{1}c_{2} < 0$. 
Then the hyperbola $H$ is contained 
in the upper left and lower right complementary regions of the two asymptotic lines. 
Now $\Phi_{1}(c_{1},c_{2},-1,1) = \Phi_{1}(c_{1},c_{2},1,-1) 
= -(1-c_{1}^{2})(1-c_{2}^{2}) < 0$. 
Therefore \\
$\Phi_{1}(c_{1}, \dots , c_{4}) < 0$ for any $-1 \leq c_{3},c_{4} \leq 1$.

Suppose that $c_{1}c_{2} > 0$. 
Then the hyperbola $H$ is contained 
in the upper right and lower left complementary regions of the two asymptotic lines. 
Let us consider $\Phi_{1}(c_{1},c_{2},1,1) = (c_{1}c_{2}-c_{1}-c_{2})^{2}-1$. 
Since $1- \sqrt{2} < c_{1},c_{2} < 0$, 
we have \\ $1 < (1-c_{1})(1-c_{2}) < 2$. 
Hence $0 < c_{1}c_{2}-c_{1}-c_{2} < 1$. 
Thus $\Phi_{1}(c_{1},c_{2},1,1) < 0$. 
Furthermore, 
$\Phi_{1}(c_{1},c_{2},c_{1},c_{2}) < 0$ by Corollary~\ref{cor:abab}. 
Therefore 
$\Phi_{1}(c_{1}, \dots , c_{4}) < 0$ 
for any $c_{1} \leq c_{3} \leq 1$ and $c_{2} \leq c_{4} \leq 1$. 
\end{proof}

\begin{rem}
\label{rem:andreev}
Andreev's theorem immediately implies that $[0,\pi/2]^{4} \subset \mathcal{A}$. 
For deformation in $\mathcal{C}$, 
cone loci with cone angles less than $\pi$ do not intersect 
unless the volumes converge to zero or a 2-dimensional Euclidean sub-cone-manifold appears, 
as shown by Kojima~\cite{kojima1998deformations}. 
\end{rem}

\begin{cor}
\label{cor:connected}
The space $\mathcal{A}$ 
is connected. 
Consequently, it holds that $\mathcal{C}_{0} = \mathcal{C}_{\mathrm{sym}}$, 
and $\frac{1}{2} \Theta \colon \mathcal{C}_{0} \to \mathcal{A}$ 
is a homeomorphism. 
\end{cor}
\begin{proof}
We show that $\mathcal{A}^{\prime}$ is path-connected. 
Take $(c_{1}, \dots , c_{4}) \in \mathcal{A}^{\prime}$. 
Without loss of generality, 
we may assume that $c_{1} \leq c_{3}$ and $c_{2} \leq c_{4}$. 
As in the proof of Theorem~\ref{thm:cube}, 
$c_{1}$ and $c_{2}$ are regarded to be fixed. 
Consider the slice 
\[
\mathcal{A}^{\prime}_{1}(c_{1},c_{2}) = 
\{(x,y) \in [c_{1},1] \times [c_{2},1] \mid 
(c_{1},c_{2},x,y) \in \mathcal{A}^{\prime} \}.
\] 
If $c_{1}c_{2} \leq 0$, 
then $\mathcal{A}^{\prime}_{1}(c_{1},c_{2}) = [c_{1},1] \times [c_{2},1]$. 
If $c_{1}c_{2} > 0$, 
then 
\[
\mathcal{A}^{\prime}_{1}(c_{1},c_{2}) = 
[c_{1},1] \times [c_{2},1] \cup \{(x,y) \in \mathbb{R}^{2} \mid 
\Phi_{1}(c_{1},c_{2},x,y) < 0 \}.
\] 
In both cases, there is a path joining 
$(c_{1}, \dots , c_{4})$ and $(c_{1},c_{2},c_{1},c_{2})$ 
by the proof of Theorem~\ref{thm:cube}. 
Moreover, 
there is a path joining 
$(c_{1},c_{2},c_{1},c_{2})$ and $(1, \dots , 1)$ 
by Corollary~\ref{cor:abab}. 
Thus we obtain a path joining $(c_{1}, \dots , c_{4})$ and $(1, \dots , 1)$. 
\end{proof}

\begin{cor}
\label{cor:global}
The global rigidity for $\mathcal{C}$ holds if and only if $\mathcal{C}_{0} = \mathcal{C}$. 
\end{cor}
\begin{proof}
If $\mathcal{C}_{0} = \mathcal{C}_{\mathrm{sym}} = \mathcal{C}$, 
the global rigidity for $\mathcal{C}$ holds 
by Theorem~\ref{thm:rigidity}. 

If $\mathcal{C}_{0} = \mathcal{C}_{\mathrm{sym}} \neq \mathcal{C}$, 
there is a non-symmetric cone structure 
$g \in \mathcal{C} \setminus \mathcal{C}_{\mathrm{sym}}$. 
Then the $\Gamma$-action on $(T^{2} \times I, L)$ 
gives distinct cone structures $g, g^{\prime} \in \mathcal{C}$ 
such that \\ $\Theta (g) = \Theta (g^{\prime})$. 
Therefore the global rigidity for $\mathcal{C}$ fails. 
\end{proof}

Finally, we prove the main theorem. 

\begin{proof}[Proof of Theorem~\ref{thm:main}]
There is $(\alpha_{1}, \dots , \alpha_{4}) \in [0,\pi)^{4}$ 
which does not belong to $\mathcal{A}$. 
Take $\max_{i} \{\alpha_{i} \} < \alpha < \pi$. 
Corollary~\ref{cor:abab} implies that $(\alpha, \alpha, \alpha, \alpha) \in \mathcal{A}$. 
While we decrease the cone angles 
from $(\alpha, \dots , \alpha)$ to $(\alpha_{1}, \dots , \alpha_{4})$, 
the trapezohedron degenerates. 
This corresponds to a degeneration in $\mathcal{C}$ 
with one or two intersections of $L_{i}$.

More explicitly, we construct two paths in $\mathcal{A}^{\prime}$ 
whose terminals correspond to degenerations. 
Firstly, let $\bm{c}_{1}(x) = (1-\sqrt{2},1-\sqrt{2},1,x \in (-1,1]^{4}$ for $0 \leq x \leq 1$. 
Then 
\begin{align*}
\Phi_{1}(\bm{c}_{1}(x)) = \Phi_{2}(\bm{c}_{2}(x)) &= 2(1-\sqrt{2})^{2} (x-1), \\
\Phi_{3}(\bm{c}_{1}(x)) &= 2x(x+1), \\
\Phi_{4}(\bm{c}_{1}(x)) &= -2(1-\sqrt{2})(\sqrt{2}x+1)(x-1). 
\end{align*}
For $0 \leq x < 1$, 
we have $\Phi_{1}(\bm{c}_{1}(t)) < 0$, $\Phi_{2}(\bm{c}_{1}(x)) < 0$, $\Phi_{4}(\bm{c}_{1}(x)) < 0$, 
and $1+x > 0$. 
Hence $\bm{c}_{1}(x) \in \mathcal{A}^{\prime}$ for $0 \leq x < 1$ 
by Theorem~\ref{thm:bound}. 
Moreover, 
$\bm{c}_{1}(1) \in \partial_{1} \mathcal{A}^{\prime}$, 
and it does not belong to the other $\partial_{i} \mathcal{A}^{\prime}$. 
In the corresponding deformation of cone structures, 
the cone angle at $L_{4}$ decreases. 
The degeneration is due to an intersection of the cone loci $L_{1}$ and $L_{2}$.

Secondly, 
let $\bm{c}_{2}(x) = \left( \dfrac{1-\sqrt{5}}{2}, \dfrac{1-\sqrt{5}}{2}, \dfrac{1-\sqrt{5}}{2}, x \right)  \in (-1,1]^{4}$ 
for $0 \leq x \leq 1$. 
Then 
\begin{align*}
\Phi_{1}(\bm{c}_{2}(x)) = \Phi_{2}(\bm{c}_{2}(x)) &= \left( \dfrac{1-\sqrt{5}}{2} \right)^{4} (x-1), \\
\Phi_{3}(\bm{c}_{2}(x)) = \Phi_{4}(\bm{c}_{2}(x)) &= \left( \dfrac{1-\sqrt{5}}{2} \right)^{2} \left( x^{2}-x-\dfrac{1+\sqrt{5}}{2} \right).
\end{align*}
They are negative for $0 \leq x < 1$. 
Hence $\bm{c}_{2}(t) \in \mathcal{A}^{\prime}$ for $0 \leq x < 1$ 
by Theorem~\ref{thm:bound}. 
Moreover, 
$\bm{c}_{2}(1) \in \partial_{1} \mathcal{A}^{\prime} \cap \partial_{2} \mathcal{A}^{\prime}$ 
as shown in the proof of Corollary~\ref{cor:double}. 
In the corresponding deformation of cone structures, 
the cone angle at $L_{4}$ decreases. 
The degeneration is due to two intersections of cone loci: 
one of them is of $L_{1}$ and $L_{2}$, 
and the other is of $L_{2}$ and $L_{3}$. 
\end{proof}

\bibliographystyle{siam}
\bibliography{ref-dcda}

\end{document}